\documentclass[a4paper]{article}
\usepackage[utf8]{inputenc}
\usepackage[T1]{fontenc}
\usepackage[english]{babel} 
\usepackage{pifont}
\newcommand{\xmark}{\ding{55}}
\usepackage{multirow}
\usepackage{mathtools}
\usepackage{amssymb,amsfonts}
\usepackage{amsthm}
\usepackage[colorlinks,breaklinks,bookmarks,plainpages=false,unicode=true]{hyperref}%
\usepackage{tikz-cd}
\usepackage{todonotes} 
\newcounter{mycomment}

\newtheorem{lem}{Lemma}[section]
\newtheorem{cor}[lem]{Corollary}
\newtheorem{prop}[lem]{Proposition}
\newtheorem{thm}[lem]{Theorem}
\newtheorem{mainthm}{Theorem}

\newtheorem{mainprop}[mainthm]{Proposition}
\theoremstyle{definition}
\newtheorem{defn}[lem]{Definition}

\newtheorem{exmp}[lem]{Example}
\DeclareMathOperator\Cayley{Cay}
\DeclareMathOperator\Ind{Ind}
\DeclareMathOperator\SL{SL}
\DeclareMathOperator\GL{GL}
\DeclareMathOperator\SO{SO}
\DeclareMathOperator\SU{SU}
\DeclareMathOperator\Sym{Sym}
\DeclareMathOperator\diag{diag}
\DeclareMathOperator\Aut{Aut}
\DeclareMathOperator\ab{ab}
\DeclareMathOperator\diam{diam}
\DeclareMathOperator\Bij{Bij}
\DeclareMathOperator\Isom{Isom}
\newcommand*{\orbite}{\mathcal O}
\newcommand*{\field}[1]{\mathbf{#1}}
\newcommand*{\category}[1]{\textbf{#1}}
\newcommand*{\PMet}{\category{PMet}}
\newcommand*{\CatS}{\category{S}}
\newcommand*{\Z}{\field{Z}}
\newcommand*{\N}{\field{N}}
\newcommand*{\R}{\field{R}}
\newcommand*{\C}{\field{C}}
\newcommand*{\BS}{B\textbf{S}}
\newcommand*{\BZ}{B\textbf{Z}}
\newcommand*{\FB}{FB\textsubscript{r}}
\newcommand*{\FH}{FH}
\newcommand*{\FW}{FW}
\newcommand*{\FA}{FA}
\newcommand*{\FR}{F\textbf{R}}
\newcommand*{\FS}{F\textbf{S}}
\newcommand{\setst}[2]{\{#1\ |\ #2\}}
\newcommand*{\powerset}[1]{\mathcal P(#1)}
\title{Wreath products of groups acting with bounded orbits}
\author{Paul-Henry Leemann\thanks{Supported by grant 200021\textunderscore188578 of the Swiss National Fund for Scientific Research.} $^{,1}$, Grégoire Schneeberger$^2$}
\date{%
{\small $^1$Institut de Math\'ematiques, Universit\'e de Neuch\^atel, 11 rue Emile-Argand, 2000~Neuch\^atel, Switzerland, \texttt{paulhenry.leemann@xjtlu.edu.cn}}\\%
{\small $^2$Section de math\'ematiques, Universit\'e de Gen\`eve, 7-9 rue du Conseil G\'en\'eral, 1205~Gen\`eve, Switzerland, \texttt{gregoire.schneeberger@unige.ch}}\\[2ex]%
\today}
\begin{document}
\maketitle
\begin{abstract}
If \CatS{} is a subcategory of metric spaces, we say that a group G has property~\BS{} if any isometric action on an \CatS-space has bounded orbits.
Examples of such subcategories include metric spaces, affine real Hilbert spaces, CAT(0) cube complexes, connected median graphs, trees or ultra-metric spaces.
The corresponding properties \BS{} are respectively Bergman's property, property~\FH{}  (which, for countable groups, is equivalent to the celebrated Kazhdan's property (T)), property~\FW{} (both for CAT(0) cube complexes and for connected median graphs), property~\FA{} and uncountable cofinality.
Historically many of these properties were defined using the existence of fixed points.

Our main result is that for many subcategories \CatS, the wreath product $G\wr_XH$ has property~\BS{} if and only if both $G$ and $H$ have property~\BS{} and $X$ is finite.
On one hand, this encompasses in a general setting previously known results for properties \FH{} and \FW.
On the other hand, this also applies to the Bergman's property.
Finally, we also obtain that $G\wr_XH$ has uncountable cofinality if and only if both $G$ and $H$ have uncountable cofinality and $H$ acts on $X$ with finitely many orbits.
\end{abstract}
%
%
%
%
%
%
%
%
%
%
%
%
%
%
%
%
%
%
%
%
%
%
%
%
%
\section{Introduction}\label{Section:Intro}
When working with group properties, it is natural to ask if they are stable under ``natural'' group operations.
One such operation, of great use in geometric group theory, is the wreath product, see Section~\ref{Section:Def} for all the relevant definitions.

An \CatS-space is a metric space with an ``additional structure'' and we will say that a group $G$ has \emph{property~\BS} if every action by isometries which preserves the structure on an \CatS-space has bounded orbits.
Formally, this means that \CatS{} is a subcategory of the category of metric spaces, and that the actions are by \CatS-automorphisms.
We note that having one bounded orbit implies that all the orbits are bounded.

In the context of properties defined by actions with bounded orbits, the first result concerning wreath products, due to Cherix, Martin and Valette and later refined by Neuhauser, concerns Kazhdan's property~(T).
\begin{thm}[\cite{MR2106770,MR2176470}] \label{T:Wreath_prop_T}
Let $G$ and $H$ be two discrete groups with $G$ non-trivial and let $X$ be a set on which $H$ acts.
The wreath product $G \wr_X H$ has property~(T) if and only if $G$ and $H$ have property~(T) and $X$ is finite.
\end{thm}
For countable groups (and more generally for $\sigma$-compact locally-compact topological groups), property (T) is equivalent, by the Delorme-Guichardet's Theorem, to property~\FH{} (every action on an affine real Hilbert space has bounded orbits), see~\cite[Thm. 2.12.4]{MR2415834}.
Hence, Theorem~\ref{T:Wreath_prop_T} can also be viewed, for countable groups, as a result on property~\FH.

The corresponding result for property~\FA{} (every action on a tree has bounded orbits) and property~\FR{} (every action on a real tree has bounded orbits), is a little more convoluted and was obtained a few years later by Cornulier and Kar.
\begin{thm}[\cite{MR2764930}]\label{Thm:FACK}
Let $G$ and $H$ be two groups with $G$ non-trivial and let $X$ be a set on which $H$ acts with finitely many orbits and without fixed points.
Then $G\wr_XH$ has property~\FA{} (respectively property~\FR) if and only if $H$ has property~\FA{} (respectively property~\FR), $G$ has no quotient isomorphic to $\Z$ and can not be written as a countable increasing union of proper subgroups.
\end{thm}
Observe that our statement of Theorem~\ref{Thm:FACK} differs of the original statement of~\cite{MR2764930}.
Indeed, where we ask $G$ to have uncountable cofinality and no quotient isomorphic to $\Z$, the authors of~\cite{MR2764930} ask $G$ to have uncountable cofinality and finite abelianization.
However, these two sets of conditions are equivalent.
One implication is trivial, as finite Abelian groups do not project onto $\Z$.
For the other implication, suppose that $G$ has uncountable cofinality but infinite abelianization $G/[G,G]$. The group $G/[G,G]$ being an infinite Abelian group, it has a countably infinite quotient $A$ --- a classical fact of which Y. de Cornulier kindly reminded us, see \cite{MR551496}[§16.11.c] for a proof. The quotient $A$ has uncountable cofinality, see Lemma~\ref{Lemma:Quotient}, and is therefore an infinite finitely generated Abelian group, which hence projects onto $\Z$.

Finally, we have an analogous of Theorem~\ref{T:Wreath_prop_T} for property~\FW{} (every action on a CAT(0) cube complexe has bounded orbits):
\begin{thm}[\cite{Cornulier2013,LS2020}]\label{Thm:PropFW}
Let $G$ and $H$ be two groups with $G$ non-trivial and let $X$ be a set on which $H$ acts.
Suppose that all three of $G$, $H$ and $G\wr_XH$ are finitely generated. Then the wreath product $G \wr_X H$ has property~\FW{} if and only if $G$ and $H$ have property~\FW{} and $X$ is finite.
\end{thm}
It is straightforward to prove that the wreath product $G \wr_X H$ is finitely generated if and only if both $G$ and $H$ are finitely generated and the number of orbits of the action of $H$ on $X$ is finite.

Theorem~\ref{Thm:PropFW} was first proved, using \emph{cardinal definite functions}, for arbitrary groups by Cornulier \cite[Propositions 5.B.3 and 5.B.4]{Cornulier2013}, but without the implication ``if $G \wr_X H$ has property~\FW, then $G$ has property~\FW''.
The authors then gave an elementary proof of it via \emph{Schreier graphs} for the specific case of finitely generated groups \cite{LS2020}.
Y. Stalder has let us know (private communication) that, using \emph{space with walls} instead of Schreier graphs, the arguments of~\cite{LS2020} can be adapted to replace the finite generation hypothesis of Theorem~\ref{Thm:PropFW} by the condition that all three of $G$, $H$ and $X$ are at most countable.
Finally, A. Genevois published in~\cite{2017arXiv170500834G} a proof of Theorem~\ref{Thm:PropFW} for wreath products of the form $G\wr_HH$, based on his \emph{diadem product of spaces}.

The above results on properties \FH, \FW{} and \FA{} were obtained with distinct methods even if the final results have a common flavor.
In the same time, all three properties \FH, \FW{}  and \FA{} can be characterized by the fact that any isometric action on a suitable metric space (respectively affine real Hilbert space, connected median graph and tree) has bounded orbits, see Definition~\ref{Def:FHFA}.
But more group properties can be characterized in terms of actions with bounded orbits.
This is, for example, the case of the Bergman's property (actions on metric spaces), the property~\FB{} (actions on reflexive real Banach spaces) or of uncountable cofinality (actions on ultrametric spaces).

By adopting the point of view of actions with bounded orbits, we obtain a unified proof of the following result; see also Theorem~\ref{Thm:Technic} for the general (and more technical) statement.
\begin{mainthm}\label{Thm:Main}
Let \BS{} be any one of the following properties: Bergman's property, property~\FB, property~\FH{} or property~\FW.
Let $G$ and $H$ be two groups with $G$ non-trivial and let $X$ be a set on which $H$ acts.
Then the wreath product $G \wr_X H$ has property~\BS{} if and only if $G$ and $H$ have property~\BS{} and $X$ is finite.
\end{mainthm}
With a little twist, we also obtain a similar result for groups with uncountable cofinality:
\begin{mainprop}\label{Prop:UncCoun}
Let $G$ and $H$ be two groups with $G$ non-trivial and let $X$ be a set on which $H$ acts.
Then the wreath product $G \wr_X H$ has uncountable cofinality if and only if $G$ and $H$ have uncountable cofinality and $H$ acts on $X$ with finitely many orbits.
\end{mainprop}
A crucial ingredient of our proofs is that the spaces under consideration admit a natural notion of Cartesian product.
In particular, some of our results do not work for trees and property~\FA, nor do they for real trees and the corresponding property~\FR.
Nevertheless, we are still able to show that if $G\wr_XH$ has property~\FA, then $H$ acts on $X$ with finitely many orbits.
Combining this with Theorem~\ref{Thm:FACK} we obtain the following theorem.
\begin{mainthm}\label{Thm:FAFiniteOrbits}
Let $G$ and $H$ be two groups with $G$ non-trivial and $X$ a set on which $H$ acts. Suppose that $H$ acts on $X$ without fixed points.
Then $G\wr_XH$ has property~\FA{} (respectively has property~\FR) if and only if $H$ has property~\FA{} (respectively has property~\FR), $H$ acts on $X$ with finitely many orbits, $G$ has no quotient isomorphic to $\Z$ and can not be written as a countable increasing union of proper subgroups.
\end{mainthm}
%
%
%
%
%
%
\paragraph{Acknowledgment}
The authors are thankful to P. de la Harpe, A. Genevois, T. Nagnibeda and A. Valette for helpful comments and to the NSF for its support.
They also gratefully  thank the anonymous referees for their valuable remarks.
%
%
%
%
%
%
%
%
%
%
%
%
%
\section{Definitions and examples}\label{Section:Def}
This section contains all the definitions, as well as some useful preliminary facts and some examples.
%
%
%
%
%
%
%
%
%
%
\subsection{Wreath products}
Let $X$ be a set and $G$ a group.
We view $\bigoplus_XG$ as the set of functions from $X$ to $G$ with finite support:
\[
	\bigoplus_XG=\setst{\varphi\colon X\to G}{\varphi(x)=1 \textnormal{ for all but finitely many }x}.
\]
This is naturally a group, where multiplication is taken componentwise.

If $H$ is a group acting on $X$, then it naturally acts on $\bigoplus_XG$
by $(h.\varphi)(x)=\varphi(h^{-1}.x)$.
This leads to the following standard definition
\begin{defn}\label{Def:WreathProd}
Let $G$ and $H$ be groups and $X$ be a set on which $H$ acts.
The \emph{(restricted \footnote{There exists an unrestricted version of this product where the direct sum is replaced by a direct product.}) wreath product} $G\wr_XH$ is the group $(\bigoplus_XG)\rtimes H$.
\end{defn}
A prominent particular case of wreath products is of the form $G\wr_HH$, where $H$ acts on itself by left multiplication.
They are sometimes called \emph{standard wreath products} or simply \emph{wreath products}, while general $G\wr_XH$ are sometimes called \emph{permutational wreath products}.
Best known example of wreath product is the so called \emph{lamplighter group} $(\Z/2\Z)\wr_\Z\Z$.
Other (trivial) examples of wreath products are direct products $G\times H$ which correspond to wreath products over a singleton $G\wr_{\{*\}}H$.
%
%
%
%
%
%
%
%
%
%
\subsection{Classical actions with bounded orbits}
In this subsection we discuss some classical group properties, which are defined by actions with bounded orbits on various metric spaces.

\paragraph{Median graphs}
For $u$ and $v$ two vertices of a connected\footnote{We will always assume that our connected graphs are non-empty. This is coherent with the definition that a connected graph is a graph with exactly one connected component.} graph $\mathcal G$, we define the total interval $[u,v]$ as the set of vertices that lie on some shortest path between $u$ and $v$.
A connected graph $\mathcal G$ is \emph{median} if for any three vertices $u$, $v$, $w$, the intersection $[u,v]\cap[v,w]\cap[u,w]$ consists of a unique vertex, denoted $m(u,v,w)$.
A graph is \emph{median} if each of its connected components is median. For more on median graphs and spaces see~\cite{MR2405677,MR2671183,MR1705337}.
If $X$ and $Y$ are both (connected) median graphs, then their Cartesian product is also a (connected) median graph.
The class of median graphs was introduced by Nebesk\'y in 1971 \cite{Neb71} and Gerasimov \cite{Ger97, MR1663779},   Roller \cite{Rol98} and Chepoï \cite{Che00} realized independently that this class coincides with the class of $1$-skeleta of CAT(0) cube complexes.
Trees are the simplest examples of connected median graphs, while the ensuing classical example show that any power set can be endowed with a median graph structure.
\begin{exmp}\label{Ex:MainMedian}
Let $X$ be a set and let  $\powerset{X}=2^X$ be the set of all subsets of~$X$.
Define a graph structure on $\powerset{X}$ by putting an edge between $E$ and $F$ if and only if $\#(E\Delta F)=1$, where $\Delta$ is the symmetric difference.
Therefore, the distance between two subsets $E$ and $F$ is $\#(E\Delta F)$ and
the connected component of $E$ is the set of all subsets $F$ with $E\Delta F$ finite.
For $E$ and $F$ in the same connected component, $[E,F]$ consists of all subsets of $X$ that both contain $E\cap F$ and are contained in $E\cup F$.
In particular, $\powerset{X}$ is a median graph, with $m(D,E,F)$ being the set of all elements belonging to at least two of $D$, $E$ and~$F$. In other words, $m(D,E,F)=(D\cap E)\cup(D\cap F)\cup(E\cap F)$.
\end{exmp}
These graphs are useful due to the following fact.
Any action of a group $G$ on a set $X$ naturally extends to an action of $G$ on $\powerset{X}$ by graph homomorphisms: $g.Y\coloneqq\setst{g.y}{y\in Y}$ for $Y\subset X$.
Note that the action of $G$ on $\powerset{X}$ may exchange the connected components.
In fact, the connected component of $E\in \powerset{X}$ is stabilized by $G$ if and only if $E$ is \emph{commensurated} by~$G$, that is if for every $g\in G$ the set $E\Delta g.E$ is finite.
\paragraph{Uncountable cofinality}
Recall that a metric space $(X,d)$ is \emph{ultrametric} if $d$ satisfies the strong triangular inequality: $d(x,y)\leq\max\{d(x,z),d(z,y)\}$ for any $x$, $y$ and $z$ in $X$.
A group $G$ has \emph{uncountable cofinality} if every action on ultrametric spaces has bounded orbits.
The following characterization of groups of countable cofinality can be extracted from~\cite{MR2240370} and we include a proof only for the sake of completeness.
It implies in particular that a countable group has uncountable cofinality if and only if it is finitely generated.
\begin{lem}\label{Lemma:CofSub}
Let $G$ be a group. Then the following are equivalent:
\begin{enumerate}
\item $G$ can be written as a countable increasing union of proper subgroups,
\item $G$ does not have uncountable cofinality, i.e. there exists an ultrametric space $X$ on which $G$ acts with an unbounded orbit,
\item There exists a $G$-invariant (for the action by left multiplication) ultrametric $d$ on $G$ such that $G\curvearrowright G$ has an unbounded orbit.
\end{enumerate}
\end{lem}
\begin{proof}
It is clear that the third item implies the second.

Let $(X,d)$ be an ultrametric space on which $G$ acts with an unbounded orbit~$G.x_0$. For any $n\in \N$ let $H_n$ be the subset of $G$ defined by
\[
	H_n\coloneqq\setst{g\in G}{d(x_0,g.x_0)\leq n}.
\]
Then $G$ is the union of the (countably many) $H_n$, which are subgroups of $G$.
Indeed, $H_n$ is trivially closed under taking the inverse, and is also closed under taking products as we have $d(x_0,gh.x_0)\leq\max\{d(x_0,g.x_0),d(g.x_0,gh.x_0)\}=\max\{d(x_0,g.x_0),d(x_0,h.x_0)\}$.
As $G.x_0$ is unbounded, they are proper subgroups.
Since they are proper subgroups and $H_n\leq H_{n+1}$, we can extract an increasing subsequence $(H_{r_n})_n$ that still satisfies $G=\bigcup_n H_{r_n}$.

Finally, suppose that $G=\bigcup_{n\in \N}H_n$, where the $H_n$ form an increasing sequence of proper subgroups.
It is always possible to suppose that $H_0=\{1\}$.
Define $d$ on $G$ by $d(g,h)\coloneqq\min\setst{n}{g^{-1}h\in H_n}$.
One easily verifies that $d$ is a $G$-invariant ultrametric. Moreover, the orbit of $1$ contains all of $G$ and is hence unbounded.
\end{proof}
%
%
%
%
%
%
%
\paragraph{Some classical group properties}
We now discuss the bounded orbits properties for actions on various classes of metric spaces.

%
%
\begin{defn}\label{Def:FHFA}
Let $G$ be a group.
It is said to have the following properties:
\begin{itemize}
\item\emph{Bergman's property} if  any action by isometries on a metric space has bounded orbits,
\item\emph{Property~\FB} if any action by affine isometries on a reflexive real Banach space has bounded orbits,
\item \emph{Property~\FH} if any action by affine isometries on a real Hilbert space has bounded orbits,
\item
\emph{Property~\FW} if any action by graph isomorphisms on a connected median graph has bounded orbits,
\item
\emph{Property~\FR} if any action by isometries on a real tree has bounded orbits,
\item
\emph{Property~\FA} if any action by graph isomorphisms on a tree has bounded orbits,
\item
\emph{Uncountable cofinality} if any action by isometries on an ultrametric space has bounded orbits.
\end{itemize}
\end{defn}
In the above, we insisted on the fact that actions are supposed to preserve the structure of the metric space under consideration, or in other words to be by automorphisms of the considered category.
This is sometimes automatic, as for example any isometry of a real Banach or Hilbert space is affine by the Mazur-Ulam theorem.
However, for cube complexes, for example, this is not the case; a $2$-regular tree has automorphism group $\Z\rtimes(\Z/2\Z)$, while its isometry group is $\R\rtimes(\Z/2\Z)$.
In other words, an isometry of a cube complex is not necessarily a cube complex isomorphism.
See also Example~\ref{Exmp:BZ}.

In the following, we will often do a slight abuse of notation and simply speak of a \emph{group action} on a space $X$, without always specifying by which kinds of maps the group acts, which should always be clear from the context.

%
%
The names \FB, \FH, \FW,  \FR{} and \FA{} come from the fact that these properties admit a description in terms of (and were fist studied in the context of) the existence of a fixed point for actions on reflexive Banach spaces, on Hilbert spaces, on spaces with walls (or equivalently on CAT(0) cube complexes, see \cite{MR2197811,MR2059193}), on real trees and on trees (\emph{Arbres} in french).
However this is equivalent with the bounded orbit property, see Proposition~\ref{Proposition:Mediane} and the discussion below~it. Observe that space with walls admit a natural pseudo-metric on them, which is not necessarily a metric. 

The Bergman's property can also be characterized via length function, see for example the beginning of~\cite{MR4119107}.

For a survey on property~\FB, see~\cite{MR3382026} and the references therein.

For countable groups (and more generally for $\sigma$-compact locally compact groups), property~\FH{} is equivalent, by the Delorme-Guichardet theorem, to the celebrated Kazhdan's property (T), but this is not true in general.
Indeed by~\cite{MR2239037} symmetric groups over infinite sets are uncountable discrete group with Bergman's property which, as we will see just below, implies property~\FH.
Such groups cannot have property (T) as, for discrete groups, it implies finite generation.

A classical result of the Bass-Serre theory of groups acting on trees~\cite{MR0476875}, is that a group $G$ has property~\FA{} if and only if it satisfies the following three conditions: $G$ has uncountable cofinality, $G$ has no quotient isomorphic to $\Z$ and $G$ is not a non-trivial amalgam.
In view of this characterization, Theorem~\ref{Thm:FACK} says that property~\FA{} almost behaves well under wreath products.
\begin{prop}\label{Prop:ImplProp}
There are the following implications between the properties of Definition~\ref{Def:FHFA}:
\[
\begin{tikzcd}[arrows=Rightarrow]
\parbox{4.5em}{\textnormal{Bergman's property}}\arrow[r]	&\textnormal{\FB} \arrow[r]	&\textnormal{\FH} \arrow[r]	\arrow[d]&\textnormal{\FW}\arrow[d] &\\%
&&\textnormal{\FR}\arrow[r]& \textnormal{\FA} \arrow[r]& \parbox{5em}{\textnormal{uncountable cofinality}}
\end{tikzcd}
\]
Moreover, except maybe for the implication $[\textnormal{Bergman's property}\implies\textnormal{\FB}]$, all implications are strict.
\end{prop}
\begin{proof}
The implications $[\textnormal{Bergman's property}\implies \textnormal{\FB}\implies\textnormal{\FH}]$ and $[\textnormal{\FW}\implies \textnormal{\FA}\Longleftarrow\textnormal{\FR}]$ trivially follow from the fact that Hilbert spaces are reflexive Banach spaces, which are themselves metric spaces and that trees are both real trees and connected median graphs.
The implication $[\textnormal{\FH}\implies \textnormal{\FW}]$ follows from the fact that a group $G$ has property~\FW{} if and only if any affine action on a real Hilbert space which preserves integral points has bounded orbits~\cite[Proposition 7.I.3]{Cornulier2013}.
The implication $[\textnormal{\FH}\implies \textnormal{\FR}]$ follows from the fact that real trees are median metric spaces, and that such spaces can be embedded into $L^1$-spaces (see for instance \cite[Theorem V.2.4]{Ver93}).
Finally, the implication $[\textnormal{\FA}\implies\textnormal{uncountable cofinality}]$ is due to Serre~\cite{MR0476875}: if $G$ is an increasing union of subgroups $G_i$, then $\bigsqcup G/G_i$ admits a tree structure by joining any $gG_i\in G/G_i$ to $gG_{i+1}\in G/G_{i+1}$.
The action of $G$ by multiplication on $\bigsqcup G/G_i$ is by graph isomorphisms and with unbounded orbits.

We now present some examples demonstrating the strictness of the implications. 
Countable groups with property~\FB{} are finite by~\cite{MR2137870}, while infinite finitely generated groups with property (T), e.g. $\SL_3(\Z)$, have property~\FH.
The group $\SL_2(\Z[\sqrt{2}])$ has property~\FW{} but not \FH, see~\cite{MR3299841}.
If $G$ is a non-trivial finite group and $H$ is an infinite group with property~\FA{} (respectively property~\FR), then $G\wr_HH$ has property~\FA{} (respectively property~\FR) by Theorem~\ref{Thm:FAFiniteOrbits}, but does not have property~\FW{} by Theorem~\ref{Thm:Main}.
The group $\Z$ has uncountable cofinality, while it acts by translations and with unbounded orbits on the infinite $2$-regular tree.
Finally, Minasyan constructed examples of groups with \FA{} but without \FR{} in~\cite{MR3465847}.
\end{proof}
The reader familiar with triangle groups
\[
	\Delta(l,m,n)=\langle a,b,c\,|\,a^2=b^2=c^2=(ab)^l=(bc)^m=(ca)^n=1\rangle\]
with $l,m,n\in\{1,2,\dots\}\cup\{\infty\}$ will be happy to observe that they provide explicit examples of groups with property~\FA{} but not property~\FW.
Indeed, if $l$, $m$ and $n$ are all three integers, then $\Delta(l,m,n)$ has property~\FA{} by Serre~\cite[Section 6.5, Corollaire 2]{MR0476875}.
And if $\kappa(l,m,n)\coloneqq\frac1l+\frac1m+\frac1n\leq 1$, then $\Delta(l,m,n)$ is the infinite symmetric group of a tilling of the Euclidean plane (if $\kappa(l,m,n)=1$) or of the hyperbolic plane (if $\kappa(l,m,n)<1$) and hence acts on a space with walls without fixed point, which implies that it does not have property FW.

In view of Proposition~\ref{Prop:ImplProp}, two questions remain open: is the implication $[\textnormal{Bergman's property}\implies\textnormal{\FB}]$ strict, and does property~\FW{} implies property~\FR?

It is possible to consider relative versions of the properties appearing in Definition~\ref{Def:FHFA}.
Let \CatS{} be any classes of metric spaces considered in Definition~\ref{Def:FHFA} and let \BS{} be the corresponding group property.
If $G$ is a group and $H$ a subgroup of~$G$, we say that the pair $(G,H)$ has \emph{relative property~\BS} if for every $G$-action on an \CatS-space, the $H$ orbits are bounded.
A group $G$ has property~\BS{} if and only if for every subgroup $H$ the pair $(G,H)$ has relative property~\BS{}, and if and only if for every overgroup $L$ the pair $(L,G)$ has relative property~\BS.
%
%
%
%
%
%
%
%
%
\subsection{Groups acting with bounded orbits on \CatS-spaces}\label{SubSec:CatS}
It is possible to define other properties in the spirit of Definition~\ref{Def:FHFA} for any ``subclass of metric spaces'', or more precisely for any subcategory of pseudo-metric spaces.
A reader not familiar with category theory and interested only in one specific subclass of metric spaces may forget all these general considerations and only verify that the arguments of Section~\ref{Section:Proof} apply for their favourite subclass of metric spaces.
\begin{defn}\label{Def:Distance}
A \emph{pseudo-metric space} is a set $X$ with a map $d\colon X\times X\to \R_{\geq0}$, called a \emph{pseudo-distance}, such that
\begin{enumerate}
\item $d(x,x)=0$ for all $x\in X$,
\item $d(x,y)=d(y,x)$ for all $x,y \in X$,
\item $d(x,z)\leq d(x,y)+d(y,z)$.
\end{enumerate}
\end{defn}
If moreover $d(x,y)\neq 0$ for $x\neq y$, the map $d$ is a \emph{distance} and $(X,d)$ is a \emph{metric space}.
On the other hand, an \emph{ultra-pseudo-metric space} is a pseudo-metric space $(X,d)$ such that $d$ satisfies the strong triangular inequality: $d(x,z)\leq \max\{d(x,y),d(y,z)\}$.
A \emph{morphism} (or \emph{short map}) between two pseudo-metric spaces $(X_1,d_1)$ and $(X_2,d_2)$ is a distance non-increasing map $f\colon X_1\to X_2$, that is $d_2(f(x),f(y))\leq d_1(x,y)$ for any $x$ and $y$ in $X_1$.
If $f$ is bijective and distance preserving, then it is an \emph{isomorphism} (or \emph{isometry}).
Pseudo-metric spaces with short maps form a category \PMet, of which the category of metric spaces (with short maps) $\category{Met}$ is a full subcategory.

If $(X,d)$ is a pseudo-metric space, we have a natural notion of the  \emph{diameter} of a subset $Y\subset X$ with values in $[0,\infty]$, defined by $\diam(Y)\coloneqq\sup\setst{d(x,y)}{x,y\in Y}$, and we say that $Y$ is \emph{bounded} if it has finite diameter.
%
%

Remind that a \emph{subcategory of \PMet}, is a category \CatS{} whose objects are pseudo-metric spaces, and whose morphisms are short maps.
The subcategory \CatS{} is \emph{full} if given two \CatS-objects $X$ and $Y$, any short map from $X$ to $Y$ is a \CatS-morphism.
A \emph{$G$-action on an \CatS-space} $X$ is simply a group homomorphism $\alpha\colon G\to\Aut_{\CatS}(X)$.
%
%

In practice, a lot of examples of (full) subcategories of \PMet{} are already subcategories of $\category{Met}$.
Obvious examples of full subcategories of $\category{Met}$ include metric spaces and ultrametric spaces (with short maps).
Affine real Hilbert and Banach spaces and more generally normed vector spaces (with affine maps) are also full subcategories of $\category{Met}$ if we restrict ourselves to morphisms that do not increase the norm (that is such that  $\|f(x)\|\leq\| x\|$).
In particular, for us isomorphisms of Hilbert and Banach spaces will always be affine isometries.

For connected graphs (and hence for its full subcategories of connected median graphs and of trees), one looks at the category $\category{Graph}$ where objects are connected simple graphs $G=(V,E)$ and where a morphism $f\colon (V,E)\to(V',E')$ is a function between the vertex sets such that if $(x,y)$ is an edge then either $f(x)=f(y)$ or $(f(x),f(y))$ is an edge.
There are two natural ways to see $\category{Graph}$ as a subcategory of $\category{Met}$.
The first one consists to look at the vertex set $V$ endowed with the shortest-path metric: $d(v,w)$ is the minimal number of edges on a path between $v$ and $w$.
The second one, consists at looking at the so called geometric realization of $(V,E)$, where each edge is seen as an isometric copy of the segment $[0,1]$.
Similarly to what happens for cube complexes (see the discussion after Definition~\ref{Def:FHFA}), the geometric realization of a graph gives an embedding $\category{Graph}\hookrightarrow\category{Met}$ which is not full.
Nevertheless, for our purpose, the particular choice of one of the above two embeddings $\category{Graph}\hookrightarrow\category{Met}$ will make no difference.

We can now formally define the group property~\BS{} as:
\begin{defn}\label{Def:PropBS}
Let \CatS{} be a subcategory of \PMet.
A group $G$ has \emph{property} \BS{} if every $G$-action by \CatS-automorphisms on an \CatS-space has all its orbits bounded.
A pair $(G,H)$ of a group and a subgroup has \emph{relative property} \BS{} if for every $G$-action by \CatS-automorphisms on an \CatS-space, the $H$-orbits are bounded.
\end{defn}
Observe that a group $G$ has property~\BS{} if and only if any $G$-action on an \CatS-space has at least one bounded orbit.

All the properties of Definition~\ref{Def:FHFA} are of the form \BS.
Another example of property of the form \BS{} can be found in~\cite[Definition 6.22]{MR1023471}: a group has property (FHyp$_{\field{C}}$) if any action on a real or complex hyperbolic space of finite dimension has bounded orbits.
This property is implied by property~\FH, but does not imply property~\FA{}~\cite[Corollary 6.23 and Example 6.24]{MR1023471}.
One can also want to look at the category of all Banach spaces (the corresponding property~BB hence stands between the Bergman's property and property~\FB), or the category of $L^p$-spaces for $p$ fixed~\cite{MR2978328} (if $p\notin\{1,\infty\}$, then B$L^p$ is implied by \FB).

Another interesting example of a property of the form \BS{} is the fact to have no quotient isomorphic to $\Z$, see Example~\ref{Exmp:BZ}.
The main interest for us of this example is that property~\FA{} is the conjonction of three properties, two of them (uncountable cofinality and having no quotient isomorphic to $\Z$) still being of the form \BS.
\begin{exmp}\label{Exmp:BZ}
Let $Z$ be the $2$-regular tree, or in other words the Cayley graph of $\Z$ for the standard generating set.
Then $\Aut_{\mathbf{Graph}}(Z)=\Z\rtimes (\Z/2\Z)$ is the infinite dihedral group and its subgroup of orientation preserving isomorphisms is isomorphic to $\Z$.
Let \textbf{S} be the category with one object $Z$ and with morphisms the orientation preserving isomorphisms.
Hence, we obtain that a group $G$ has no quotient isomorphic to $\Z$ if and only if every $G$-action on \textbf{S}-space has bounded orbits.
Let us denote by \BZ{} this property.

Since $Z$ is a tree, property~\FA{} implies property~\BZ. This implication is strict as demonstrated by $\field{Q}$.
In fact, the counterexample $\field{Q}$ shows that \BZ{} does not imply uncountable cofinality, while $\Z$ demonstrate that uncountable cofinality does not implies \BZ.
\end{exmp}
An example of an uninteresting property~\BS{} is given by taking \CatS{} to be the category of metric spaces of finite diameter (together with short maps), in which case any group has \BS.
On the other hand, if \CatS{} is the category of extended pseudo-metric spaces ($d$ takes values in $\R_+\cup\{\infty\})$, only the trivial group has \BS.
Indeed, one can put the extended metric $d(x,y)=\infty$ if $x\neq y$ on $G$ and then the action by left multiplication of $G$ on $(G,d)$ is transitive and with an unbounded orbit as soon as $G$ is non trivial.

The category \PMet{} has the advantage (over $\category{Met}$) of behaving more nicely with respect to categorical constructions and quotients.
However, we have the following lemma.
\begin{lem}\label{Lemma:BergQPMet}
A group $G$ has Bergman's property~(respectively uncountable cofinality) if and only if any isometric $G$-action on a pseudo-metric (respectively ultra-pseudo-metric) space has bounded orbits.
\end{lem}
\begin{proof}
One direction is trivial.
For the other direction, let $(X,d)$ be a pseudo-metric space on which $G$ acts by isometries.
Let $X'\coloneqq X/\sim$ be the quotient of $X$ for the relation $x\sim y$ if $d(x,y)=0$ and let $d'$ be the quotient of $d$.
Then $(X',d')$ is a metric space, the action of $G$ passes to the quotient and $G.x$ is $d$-bounded if and only if $G.[x]$ is $d'$-bounded.
Finally, if $d$ satisfies the strong triangular inequality, then so does~$d'$.
\end{proof}
On the other hand, the following result is perhaps more surprising.
\begin{lem}
A group $G$ has Bergman's property if and only if any $G$-action by graph automorphisms on a connected graph has bounded orbits.
\end{lem}
\begin{proof}
The left-to-right implication is clear.

%
For the other direction, let $X$ be a metric space. Let $G(X)$ denotes the graph obtained from a vertex-set $X$ by applying the following process: for any two $x, y\in X$ add a path of length $\lfloor d(x, y)\rfloor + 1$ between $x$ and $y$. $G(X)$ is connected and the obvious inclusion $\iota\colon X\to G(X)$ is a quasi-isometric embedding. Moreover, the construction is canonical, so every group action on $X$ extends to a group action on $G(X)$, making $\iota$ equivariant. So if a group satisfying the bounded orbit property on connected graphs acts on a metric space $X$, then its induced action on $G(X)$ has bounded orbits, which implies that its orbits in $X$ are bounded.
\end{proof}
Observe that an alternative proof of the above Lemma can be easily deduced from the following characterization of Bergman's property due to Cornulier~\cite{MR2240370}: a group $G$ has Bergman's property if and only if it has uncountable cofinality and for every generating set $T$ of $G$, the Calyey graph $\Cayley(G;T)$ is bounded.
Details are left to the reader.

While we will be able to obtain some results for a general subcategory of \PMet, we will sometimes need to restrict ourselves to subcategories  satisfying good properties. In particular, we will use three axioms: one on the existence of non-trivial $G$-action, one on the existence of finite Cartesian powers and one on infinite Cartesian powers.
Our Cartesian powers will need to be in some sense compatible with the bornology, but the conditions for finite and infinite powers will not be the same.
A summary of which axioms are satisfied by the above mentioned subcategories of \PMet{} can be found in page~\pageref{Table1}.
\begin{defn}[Axiom (A1)]\label{Def:NonTrivialAction}
A subcategory \CatS{} of \PMet{} has \emph{non-trivial group actions} if for every non-trivial group $G$ there exists an \CatS-space $X$ and an action $G\curvearrowright X$ by \CatS-automorphisms moving at least one point.
\end{defn}
Examples of categories \CatS{} with non-trivial group actions include (ultra-) metric spaces and metric spaces of finite diameter (with the action by multiplication of $G$ on itself, endowed with the discrete metric), (reflexive) Banach spaces and Hilbert spaces (with $X=l^2(G)$ and the permutation action of $G$), $L^p$ spaces ($G$ acting by permutation on $\ell^p(G)$) and finally connected median graphs and (real) trees ($X$ has one central vertex to which we attach an edge for every $g\in G$ and the action of $G$ is by left multiplication).
On the opposite side, both $Z$ from Example~\ref{Exmp:BZ} and real and complex finite dimensional hyperbolic spaces do not have non-trivial group actions. Indeed, a group acts non-trivially on $Z$ if and only if it projects onto $\Z$.
For hyperbolic spaces, a group $G$ acts on a hyperbolic space of dimension $n$ if and only if it projects onto a subgroup of  $\SO(n,1)$ or $\SU(n,1)$. In particular, if the action is non-trivial, then $G$ projects onto a non-trivial subgroup of $\GL_n(\C)$, whose all finitely generated subgroups are residually finite.
We conclude that a finitely generated infinite simple group $G$ does not admit a non-trivial action on a real or complexe hyperbolic space of finite dimension.

Before introducing the axioms about Cartesian powers, let us recall the definition and some properties of the product distances $d_p$.
\begin{defn}
For a real $p\geq 1$ and a collection of non-empty metric spaces $(X_i,d_i)_{i\in I}$, we have the maps
\begin{align*}
d_p\colon \prod_{i\in I}X_i\times \prod_{i\in I}X_i&\to \R\\
\bigl(f,g)&\mapsto\Bigl(\sum_{i\in I}d_i(f(i),g(i))^p\Bigr)^{\frac1p}
\end{align*}
and
\begin{align*}
d_\infty\colon \prod_{i\in I}X_i\times \prod_{i\in I}X_i&\to \R\\
\bigl(f,g)&\mapsto\sup_{i\in I}d_i(f(i),g(i)),
\end{align*}
with the convention that a sum with uncountably non-zero summands is infinite.
\end{defn}
If $I$ is finite, then $d_p$ is a distance on $\prod_{i\in I}X_i$. Moreover, it is compatible with the bornology in the sense that if $E\subseteq X$ is unbounded, then the diagonal $\diag(E)\subset X^n$ is also unbounded.
The following definition generalizes this comportement to other distances.
\begin{defn}[Axiom (A2)]\label{Def:Cartesian}
A subcategory \CatS{} of \PMet{} satisfies \emph{axiom (A2)} if for any \CatS-space $X$ and any integer $n$, there exists an \CatS-object, called a \emph{$n$\textsuperscript{th} Cartesian power of $X$} and written $X^n$, such that:
\begin{enumerate}
\item
As a set, $X^n$ is the $n$\textsuperscript{th} Cartesian power of $X$,
\item\label{Item:Product}
The canonical image of $\Aut_{\CatS}(X)^n\rtimes \Sym(n)$ in $\Bij(X^n)$ lies in $\Aut_{\CatS}(X^n)$,
\item\label{Condidef:2}
If $E\subset X$ is unbounded, then the diagonal $\diag(E)\subset X^n$ is unbounded.
\end{enumerate}
\end{defn}
A good heuristic is that your favourite subcategory of \PMet{} would satisfies axiom (A2) in the above sense if and only if it already has a classical operation which is called \emph{Cartesian product}.
An \CatS-object satisfying the first two properties of Definition~\ref{Def:Cartesian} will be called a finite Cartesian power.

For metric spaces (of finite diameter), the categorical product (corresponding to the metric $d_\infty=\max\{d_X,d_Y\}$) works fine, but any product metric of the form $d_p=(d_X^p+d_Y^p)^{\frac1p}$ for $p\in[1,\infty]$ works as well.
For ultra-metric spaces, the categorical product $d_\infty$ works fine.
For $L^p$ spaces (and hence for Hilbert spaces) we take the usual Cartesian product (which is also the categorial product), which corresponds to the metric $d_p$.
For (reflexive) Banach spaces, any product metric of the form $d_p$ works.
For connected graphs, the usual Cartesian product, which corresponds to $d_1=d_X+d_Y$ works well, but one can also take the strong product (which is the categorial product in \category{Graph}), that is the distance $d_\infty$.
For connected median graphs, only the usual Cartesian product with $d_1$ works.\footnote{The category of median graphs does not have categorial products.}
On the other hand, (real) trees, $Z$ from Example~\ref{Exmp:BZ} and hyperbolic spaces do not have finite Cartesian powers and hence cannot satisfy axiom (A2).

As the above examples illustrate, there can be multiple non-isomorphic spaces playing the role of $X^n$. As our results will not depend on a particular choice of a Cartesian power, we will sometimes make a slight abuse of language and speak of \emph{the} Cartesian power $X^n$.

The situation for infinite products is more complicated. Indeed, if $I$ is infinite then the map $d_p$ is not necessary a distance on $\prod_{i\in I}X_i$ as it can take infinite values.
The solution consists of looking at the subset of $\prod_{i\in I}X_i$ on which $d_p$ takes finite values.
Formally we first need to chose a base-point $x_i$ in $X_i$ for each $i\in I$, which gives us an element $f_0\in \prod_{i\in I}X_i$ defined by $f_0(i)=x_i$. We can then define
\begin{equation*}
\sideset{}{}\bigoplus_{i \in I}X_i \coloneqq\setst{f\in \prod_{i\in I}X_i}{f(i)=f_0(i)\textnormal{ for all but finitely many }i}
\end{equation*}
and
\begin{equation*}
\sideset{}{^p}\bigoplus_{i \in I}X_i \coloneqq\Bigl\{f\in \prod_{i\in I}X_i\,\Big|\,\sum_i d_i\bigl(f(i),f_0(i)\bigr)^p<\infty\Bigr\}.
\end{equation*}
A priori, the above definitions depend on the choice of the $x_i$. However, since our results will not depend on a particular choice of base-points, we will omit to specify it.
Moreover, if $\Aut_\CatS(X)$ acts transitively on $X$, then $\bigoplus_{i \in I}^pX$ will not depend on the choice of $x\in X$. If all the $X_i$ are equal to $\R$ with $x_i =0$, then $\bigoplus_{i \in I}^pX_i$ is the classical Banach space $\ell^p(I)$ while $\bigoplus_{i \in I}X_i$ is the (non-complete) sequence space $c_{00}(I)$.

It is straightforward to verify that $\bigoplus_{i\in I} X_i\subseteq\bigoplus_{i\in I}^p X_i\subseteq\prod_{i\in I}X_i$ and that the first inclusion is an equality if the $X_i$ are uniformly discrete with an uniform lower bound on their packing radius.
Moreover the restriction of the map $d_p$ to $\bigoplus_{i\in I}^p X_i$ is a distance.
While $(\bigoplus_{i\in I} X_i, d_p)$ is a metric space, it is in general not complete even when the $X_i$ are complete, which is why we needed to define $\bigoplus_{i\in I}^p X_i$; this will be important for Banach and Hilbert spaces.

One common feature of the product distances $d_p$ for $p\neq\infty$ is that, in some rough way, they are able to detect the number of coordinates on which $f$ differs from $f_0$.
Our last axiom will generalize this comportment.
\begin{defn}[Axiom (A3)]\label{Def:InfiniteCartesian}
A subcategory \CatS{} of \PMet{} satisfies \emph{axiom (A3)} if for any \CatS-space $X$, any element $x_0$ of $X$ and any infinite set $I$, there exists an \CatS-object, called the \emph{$I$\textsuperscript{th} Cartesian power of $X$} and written $\bigoplus_I ^\CatS X$, such that:
\begin{enumerate}
\item
As sets we have the inclusions $\bigoplus_I X\subseteq \bigoplus_I^\CatS X\subseteq \prod_I X$,
\item
The canonical image of $\Aut_{\CatS}(X)\wr_I \Sym(I)$ in $\Bij(\prod_I X)$ lies in $\Aut_{\CatS}(\bigoplus_I^\CatS X)$,
\item
For any $y$ in $X$ with $d(y,x_0)>0$, the following set has infinite diameter
\[
	\Bigl\{f\in \bigoplus_I X\,\Big|\,f(i)=y \textnormal{ for finitely many } i \textnormal{ and otherwise }f(i)=x_0\Bigr\}.
\]
\end{enumerate}
\end{defn}
An object satisfying the first two items of Definition~\ref{Def:InfiniteCartesian} will be called an infinite Cartesian power.

In practice, the above definition is often easy to verify.
Indeed, in most cases when \CatS{} has finite Cartesian powers it is for some product metric of the form $d_p$.
Then the metric space $(\bigoplus^p_IX,d_p)$ will usually be an infinite Cartesian power in \CatS{} and, if $p\neq\infty$, it will satisfies (A3).
In particular, (A3) holds in the following categories: metric spaces, (reflexive) Banach spaces (with $d_p$ for $1<p<\infty$), Hilbert spaces and $L^p$ spaces if $p\neq\infty$, connected (median) graphs.
On the other hand, (real) trees, hyperbolic spaces and $Z$ from Example~\ref{Exmp:BZ} do not have a sensible notion of infinite Cartesian powers. Finally, while ultra-metric spaces, $L^\infty$ spaces and spaces of finite diameter have infinite Cartesian powers (for $d_\infty$), axiom (A3) does not hold as the diameter of the set appearing in Definition~\ref{Def:InfiniteCartesian} is $d(y,x_0)$.

Finally, we introduce one last definition
\begin{defn}
A subcategory \CatS{} of \PMet{} has \emph{bornological Cartesian powers} if it satisfies both axiom (A2) and (A3).
\end{defn}
\subsection{Variations and generalizations}
This subsection is devoted to variations and generalizations of property~\BS.
It is intended as a note for the interested reader, and can be skipped without any harm.
\paragraph{Groups acting with fixed point on \CatS-spaces}
Some of the properties that are of interest for us have been historically defined via the existence of a fixed point for some action.
More generally, we say that a group $G$ has \emph{property~\FS} if any $G$-action on an \CatS-space has a fixed point.

Since our actions are by isometries, property~\FS{} implies property~\BS. The other implication holds as soon as we have a suitable notion of the center of a (non-empty) bounded subset $X$.
For a large class of metric spaces, this is provided by the following result of Bruhat and Tits:
\begin{prop}[{\cite[Chapter 3.b]{MR1023471}}]\label{Proposition:Mediane}
Let $(X,d)$ be a complete metric space such that the following two conditions are satisfied:
\begin{enumerate}
\item For all $x$ and $y$ in $X$, there exists a unique $m\in X$ (the middle of $[x,y]$) such that $d(x,m)=d(y,m)=\frac12d(x,y)$;
\item For all $x$, $y$ and $z$ in $X$, if $m$ is the middle of $[y,z]$ we have the median's inequality  $2d(x,m)^2+\frac12d(y,z)^2\leq d(x,y)^2+d(x,z)^2$.
\end{enumerate}
Then if $G$ is a group acting by isometries on $X$ with a bounded orbit, it has a fixed point.
\end{prop}
Examples of complete metric spaces satisfying Proposition~\ref{Proposition:Mediane} include among others: Hilbert spaces, Bruhat-Tits Buildings, Hadamard spaces (i.e. complete CAT(0) spaces and in particular CAT(0) cube complexes which are either finite dimensional or locally finite), trees and $\R$-trees; with the caveat that for trees and $\R$-trees, the fixed point is either a vertex or the middle of an edge.
See~\cite[Chapter 3.b]{MR1023471} and the references therein for more on this subject.
On the other hand,~\cite[Lemma 2.2.7]{MR2415834} gives a simple proof of the existence of a center for bounded subsets of Hilbert spaces, and more generally of reflexive Banach spaces, but this also directly follows from the Ryll-Nardzewski fixed-point theorem.
Finally, properties \FS{} and \BS{} are equivalent for the category of separable  uniformly  convex Banach  spaces by the existence of the Chebyshev center of a (nonempty) bounded set.

For action on metric spaces or on connected median graphs, \FS{} is strictly stronger than \BS.
Indeed, this trivially follows from the action by rotation of $\Z/4\Z$ on the square graph.
However,  by~\cite{MR1663779,Rol98} if a group $G$ acts on a connected median graph with a bounded orbit, then it has a finite orbit.
For actions on an ultra-metric spaces \FS{} is strictly stronger than \BS{} since the finite group  $\Z/2\Z$ acts without fixed point on the Cantor space $X\subset[0,1]$ by $x\mapsto 1-x$.
Property~\FS{} is also strictly stronger than \BS{} for \CatS{} the category of all Banach spaces. Indeed, by \cite{MR2656670} any infinite discrete group admits an action without fixed point on some Banach space and hence does not have property~\FS, while there exists infinite groups with the Bergmann's property which implies property~\BS.
\paragraph{Actions with uniformly bounded orbits}
One might wonder what happens if in Definition~\ref{Def:FHFA} we replace the requirement of having bounded orbits by having uniformly bounded orbits.
It turns out that this is rather uninteresting, as a group $G$ is trivial if and only if any $G$-action on a metric space (respectively on an Hilbert space, on a connected median graph, on a tree or on an ultrametric space) has uniformly bounded orbits.
Indeed, if $G$ is non-trivial, then, for the action of $G$ on the Hilbert space $\ell^2(G)$ the orbit of $n\cdot \delta_g$ has diameter $n\sqrt2$.
For a tree (and hence also for a connected median graph), one may look at the tree $T$ obtained by taking a root $r$ on which we glue an infinite ray for each element of~$G$.
Then $G$ naturally acts on $T$ by permuting the rays.
The orbits for this action are the $\mathcal L_n=\setst{v}{d(v,r)=n}$ which have diameter $2n$.
Finally, it is possible to put an ultradistance on the vertices of $T$ by  $d_\infty(x,y)\coloneqq\max\{d(x,r),d(y,r)\}$ if $x\neq y$.
Then the orbits are still the~$\mathcal L_n$, but this time with diameter $n$.
\paragraph{Topological groups}
One can wonder what happens for topological groups. While, the wreath product of topological groups is not in general a topological group, this is the case if $G$ is discrete and $X$ is a discrete set endowed with a continuous $H$-action.
In this particular context, Theorem~\ref{Thm:Technic}, as well as its proof, remains true.
The details are left to the interested reader.
\paragraph{Categorical generalizations}
In the above, we defined property~\BS{} for \CatS{} a subcategory of \PMet.
It is possible to generalize this definition to more general categories.
We are not aware of any example of the existence of a group property arising in this general context that is not equivalent to a property~\BS{} in the sense of Definition~\ref{Def:PropBS}, but still mention it for the curious reader.

On one hand, we can replace \PMet{} with a more general category.
For example, one can look at the category $\category{M}$ of sets $X$ endowed with a map $d\colon X\times X\to\R_{\geq 0}$ satisfying the triangle inequality.
That is, $d$ is a pseudo-distance, except that is is not necessary symmetric and $d(x,x)$ may be greater than $0$.
All the statements and the proofs remain true for \CatS{} a subcategory of $\category{M}$.

On the other hand, we can define property~\BS{} for any category \CatS{} over \PMet, that is for any category~\CatS{} endowed with a faithful functor $F\colon\CatS\to \PMet$.
Such a couple $(\CatS,F\colon\CatS\to \PMet)$ is sometimes called a \emph{structure over \PMet}, and $F$ is said to be \emph{forgetful}.
In this context, we need to be careful to define Cartesian powers (Definitions~\ref{Def:Cartesian} and~\ref{Def:InfiniteCartesian}) using $F$, but apart for that all the statements and all the proofs remain unchanged.
An example of such an \CatS{} that cannot be expressed as a subcategory of \PMet{} is the category of edge-labeled graphs, where the morphisms are graph morphisms that induce a permutation on the set of labels.
However, in this case the property~\BS{} is equivalent to the Bergman's property.

One can also combine the above two examples and look at couples $(\CatS,F\colon\CatS\to \category{M})$, with $F$ faithful.

Finally, in view of Definitions~\ref{Def:PropBS}, \ref{Def:Cartesian} and ~\ref{Def:InfiniteCartesian}, the reader might ask why we are working in \PMet{} or $\category{M}$ instead of $\category{Born}$, the category of bornological spaces together with bounded maps.
The reason behind this is the forthcoming Lemma~\ref{Lemma:InterProd} and its corollaries, which fail for general bornological spaces.
In fact, all the statements and the proofs remain true for a general $(\CatS,F\colon\CatS\to \category{Born})$ as soon as \CatS{} satisfies Lemma~\ref{Lemma:InterProd}.
Here is an example of such an \CatS{} which does not appear as a category over \category{M}.
Take $\kappa$ to be an infinite cardinal and let $\mathbf{S}_\kappa$ be the subcategory of $\category{Born}$ where a subset $E$ of a $\mathbf{S}_\kappa$-space is bounded if and only if $\lvert E\rvert<\kappa$.
A group $G$ has property B$\mathbf{S}_\kappa$ if and only if $\lvert G\rvert<\kappa$.
\section{Proofs of the main results}
\label{Section:Proof}
Throughout this section, \CatS{} will denote a subcategory of \PMet{} and \BS{} the group property~\emph{every action by \CatS-automorphisms on an \CatS-space has bounded orbits}.
In Subsection~\ref{SubSec:CatS}, we defined $3$ axioms that \CatS{} might satisfy. Axiom (A1) simply states that a non-trivial group acts non-trivially on some \CatS-space. Axioms (A2) and (A3) guarantee the existence of finite and infinite Cartesian powers, which should be compatible in some sense with the bornology. Finally, \CatS{} has bornological Cartesian powers if it satisfy both axioms (A2) and (A3).
In Table~\ref{Table1} we present a short reminder on whenever these axioms are satisfied for some subcategories of \PMet{} that were mentioned in Sections~\ref{Section:Intro} and~\ref{Section:Def}.
\begin{table}
\begin{tabular}{*{2}{|c}|*{3}{|c}|}
\hline
\multirow{2}{*}{Category \CatS}&Corresponding&\multicolumn{3}{c|}{satisfies axiom}\\\cline{3-5}
  &group property & (A1) & (A2) & (A3)\\
\hline
\hline
Metric spaces	&Bergmann's property&\checkmark&\checkmark&\checkmark\\
\hline
Banach spaces			&BB			&\checkmark&\checkmark&\checkmark\\
\hline
Reflexive Banach spaces	&\FB		&\checkmark&\checkmark&\checkmark\\
\hline
$L^p$ spaces ($p$ fixed)	&B$L^p$	&\checkmark&\checkmark		  &iff $p\neq\infty$\\
\hline
Hilbert spaces			&\FH		&\checkmark&\checkmark&\checkmark\\
\hline
$\R$ and $\C$ hyperbolic spaces&FHyp$_{\field{C}}$&\xmark&\xmark&\xmark\\
\hline
Median graphs			&\FW		&\checkmark&\checkmark&\checkmark\\
\hline
Real trees				&\FR		&\checkmark&\xmark	  &\xmark\\
\hline
Trees					&\FA		&\checkmark&\xmark	  &\xmark\\
\hline
Ultrametric spaces	&uncountable cofinality	&\checkmark&\checkmark&\xmark\\
\hline
Z$=2$-regular tree with $\Isom^+$	&B$\Z$	&\xmark	&\xmark	&\xmark\\
\hline
Spaces of finite diameter&hold for all groups		&\checkmark&\checkmark&\xmark\\
\hline
\end{tabular}
\caption{Axioms for category \CatS. (A1) = non-trivial group actions (Definition~\ref{Def:NonTrivialAction}), (A2) = Definition~\ref{Def:Cartesian}, (A3) = Definition~\ref{Def:InfiniteCartesian}, bornological Cartesian powers = (A2)+(A3).}
\label{Table1}
\end{table}

The main result of this section is the following theorem that implies Theorem~\ref{Thm:Main}.
\begin{thm}\label{Thm:Technic}
Suppose that \CatS{} has non-trivial group actions and bornological  Cartesian powers.
Let $G$ and $H$ be two groups with $G$ non-trivial and let $X$ be a set on which $H$ acts. Then the wreath product $G \wr_X H$ has property~\BS{} if and only if $G$ and $H$ have property~\BS{} and $X$ is finite.
\end{thm}
Theorem~\ref{Thm:Technic} is a direct consequence of the forthcoming Corollary~\ref{Cor:Wreath} and Lemmas~\ref{Lemma:XFinite} and~\ref{Lemma:Unboundedness}.
The conclusion of Theorem~\ref{Thm:Technic} remains true if the hypothesis on \CatS{} are replaced  by ``\CatS{} satisfies (A2) and property \BS{} implies property \FW'', see the discussion after Lemma~\ref{Lemma:XFinite} for more details.

We now state two elementary but useful results.
\begin{lem}\label{Lemma:Quotient}
Let $G$ be a group and $H$ be a quotient.
If $G$ has property~\BS, then so has $H$.
\end{lem}
\begin{proof}
If $H$ acts on some \CatS-space $X$ with an unbounded orbit, then the surjection $G\twoheadrightarrow H$ gives us a $G$-action on~$X$, with the same orbits as the $H$-action.
\end{proof}
\begin{lem}\label{Lemma:InterProd}
Let $G$ be a group and $A$ an $B$ be two subgroups such that $G=AB$.
If both $(G,A)$ and $(G,B)$ have relative property~\BS, then $G$ has property~\BS.
\end{lem}
\begin{proof}
Let $X$ be an \CatS-space on which $G$ acts 
and let $x$ be an element of $X$.
Let $D_1$ be the diameter of $A.x$ and $D_2$ the diameter of $B.x$. By assumption, they both are finite.
Since $A$ acts by isometries, all the $a.Bx$ have diameter $D_2$.
Let $y$ be an element of $G.x$. There exists $a\in A$ such that $y$ belongs to $a.Bx$. Since $1$ belongs to $B$, $y$ is at distance at most $D_2$ of $a.x$ and hence at distance at most $D_1+D_2$ of $x$.
Therefore, the diameter of $G.x$ is finite.
\end{proof}
By combining Lemmas~\ref{Lemma:Quotient} and~\ref{Lemma:InterProd}, we obtain the following three corollaries on direct, semi-direct and wreath products.
\begin{cor}\label{Cor:Prod}
Let $G$ and $H$ be two groups. Then $G\times H$ has property~\BS{} if and only if both $G$ and $H$ have property~\BS.
\end{cor}
\begin{cor}\label{Cor:Semidirect}
Let $N\rtimes H$ be a 
semidirect product. Then
\begin{enumerate}
\item
If $N\rtimes H$ has property~\BS, then so has $H$,
\item
If both $N$ and $H$ have property~\BS, then $N\rtimes H$ also has property~\BS.
\end{enumerate}
\end{cor}
\begin{cor}\label{Cor:Wreath}
Let $G$ and $H$ be two groups and $X$ a set on which $H$ acts.
Then,
\begin{enumerate}
\item
If $G\wr_X H$ has property~\BS, then so has $H$,
\item
If both $G$ and $H$ have property~\BS{} and $X$ is finite, then $G\wr_X H$ has property~\BS.
\end{enumerate}
\end{cor}
When \CatS{} has a suitable notion of quotients (by a group of isometries), it is possible to obtain a strong version of Lemmas~\ref{Lemma:Quotient} and~\ref{Lemma:InterProd}.
Here is the corresponding result for Bergman's property and uncountable cofinality.
\begin{prop}\label{Prop:Extension}
Let \BS{} be either Bergman's property or the property of having uncountable cofinality.
Let $1\to N\to G\to H\to 1$ be a group extension.
Then $G$ has property~\BS{} if and only if $H$ has property~\BS{} and the pair $(G,N)$ has relative \BS{} property.
\end{prop}
\begin{proof}
One direction is simply Lemma~\ref{Lemma:Quotient} and the definition of relative property~\BS.

On the other hand, let $(X,d)$ be a pseudo-metric space on which $G$ acts 
by isometries and let $x$ be an element of $X$.
Let $\setst{g_i}{i\in I}$ be a transversal for $N$, that is $H\cong\setst{g_iN}{i\in I}$ with the quotient multiplication.
By assumption, $N.x$ is bounded of diameter $D_1$ and for any $i\in I$ the subset $g_iN.x$ of $X$ has also diameter $D_1$.
Since $N$ is a subgroup of isometries of $X$, the map $d'\colon X/N\times X/N\to\R$ defined by $d'([x],[y])\coloneqq\inf\setst{d(x',y')}{x'\in N.x,y'\in N.y}$ is the quotient pseudo-distance on $X/N$.
Indeed, while the map $d'$ might not satisfies the triangle inequality for a generic quotient $X/\sim$, this is the case if the quotient is by a subgroup of isometries; details are left to the reader.
Moreover, if $d$ satisfies the strong triangle inequality, then so does $d'$.
The quotient action of $H\cong G/N$ on $X/N$ is 
by isometries and the diameter of $H.xN$ is bounded, say by $D_2$.
In particular, for any $i$ and $j$ in $I$, the distance between the subsets $g_iN.x$ and $g_jN.x$ of $X$ is bounded by $D_2$.
Since this distance is an infimum, there exist actual elements of $g_iN.x$ and $g_jN.x$ at distance less than $D_2+1$.
Altogether, we obtain that any $y$ in $G.x$ is at distance at most $D_1+D_2+1$ of~$x$.
Hence, the orbit $G.x$ is bounded.
\end{proof}
Since the triangle graph, which is not median, is a quotient of the $2$-regular infinite tree by a subgroup of isometries, the proof of Proposition~\ref{Prop:Extension} does not carry over for properties \FW{} and \FA.
Similarly, the quotient of $\R$ by the action of $\Z/2\Z$ given by $x\mapsto-x$ is not a Banach space and hence the proof of Proposition~\ref{Prop:Extension} does not apply to properties \FH{} and \FB.
However, the statement of Proposition~\ref{Prop:Extension} (stability under extensions) remains true for properties \FH, \FB, \FW{} and \FA.
For properties \FH{} and \FB, this follows from the fixed-point definition and the fact that a non-empty closed subset of an Hilbert space (respectively of a reflexive Banach space) is an Hilbert space (respectively a reflexive Banach space) itself.
For property~\FW{} and \FA, see~\cite[Proposition 5.B.3]{Cornulier2013} and~\cite{MR0476875}.

We now state a result on infinite direct sums.
\begin{lem}\label{Lemma:Cof}
Let $G$ and $(G_x)_{x\in X}$ be non-trivial groups and let $H$ be a group acting on $X$.
 Then
\begin{enumerate}
\item $\bigoplus_{x\in X}G_x$ has uncountable cofinality if and only if all the $G_x$ have uncountable cofinality and $X$ is finite,
\item If $G\wr_XH$ has uncountable cofinality, then $H$ acts on $X$ with finitely many orbits.
\end{enumerate}
\end{lem}
It is of course possible to prove Lemma~\ref{Lemma:Cof} using the characterization of uncountable cofinality in terms of subgroups, in which case the proof is a short exercise left to the reader. However, we find enlightening to prove it using the characterization in terms of actions on ultrametric spaces.
\begin{proof}[Proof of Lemma~\ref{Lemma:Cof}.]
One direction of the first assertion is simply Corollary~\ref{Cor:Prod} and holds for any property of the form \BS.
For the other direction, for any \CatS, if $\bigoplus_{x\in X}G_x$ has property~\BS{} then all its quotients, and hence all the $G_x$, have property~\BS.
Hence, we have to prove that an infinite direct sum of non-trivial groups does not have uncountable cofinality.
If $X$ is infinite, there exists a countable subset $Y \subset X$. Let $Z \coloneqq X\setminus Y$, thus we have $X = Y \sqcup Z$. We can decompose the direct sum as $\bigoplus_{x\in X}G_x = (\bigoplus_{x\in Y}G_x) \times (\bigoplus_{x\in Z}G_x)$ and then, by Corollary~\ref{Cor:Prod}, if $\bigoplus_{x\in Y}G_x$ does not have uncountable cofinality, then neither does $\bigoplus_{x\in X}G_x$.
So let us fix an enumeration of $Y$ and let $K\coloneqq \bigoplus_{i\geq 1}G_i$ and for each $i$, choose $g_i\neq 1$ in $G_i$.
Let $d_{\max}(f,g)\coloneqq\max\setst{i}{f(i)\neq g(i)}$. 
This is a $K$-invariant ultra-metric for the action by left multiplication of $G$ on itself.
Then for every integer $n$, the orbit $K.\{1,1,\dots\}$ contains $\{g_1,\dots,g_n,1,\dots\}$ which is at distance $n$ of $\{1,1,\dots\}$ for $d_{\max}$ if the $g_i$ are not equal to $1$.

The second assertion is a simple variation on the first.
Indeed, we have
\[
	G\wr_XH\cong(\bigoplus_{Y\in X/H}L_Y)\rtimes H\qquad\textnormal{with}\qquad L_Y\cong\bigoplus_{y\in Y}G_y,
\]
where $X/H$ is the set of $H$-orbits.
The important fact for us is that $H$ fixes the decomposition into $L_Y$ factors: for all $Y$ we have $H.L_Y=L_Y$.
Up to regrouping some of the $L_Y$ together we hence have $G\wr_XH\cong\bigl(\bigoplus_{i\geq 1}L_i\bigr)\rtimes H$ with $H.L_i=L_i$ for all $i$.
Now, we have an ultradistance $d_{\max}$ on $L\coloneqq\bigoplus_{i\geq 1}L_i$ as above and we can put the discrete distance $d$ on $H$.
Then $d_{\max}'\coloneqq\max\{d_{\max},d\}$ is an ultradistance on $\bigl(\bigoplus_{i\geq 1}L_i\bigr)\rtimes H$, which is $\bigl(\bigoplus_{i\geq 1}L_i\bigr)\rtimes H$-invariant (for the action by left multiplication).
From a practical point of view, we have $d'_\infty\bigl((\varphi,h),(\varphi',h')\bigr)\coloneqq\max\setst{i}{\varphi(i)\neq \varphi'(i)}$ if $\varphi\neq \varphi'$ and $d'_{\max}\bigl((\varphi,h),(\varphi,h')=1$ if $h\neq h'$.
Since the action of $L$ on itself has an unbounded orbit for $d_{\max}$, the action of $\bigl(\bigoplus_{i\geq 1}L_i\bigr)\rtimes H$ on itself has an unbounded orbit for $d'_{\max}$.
\end{proof}
We directly obtain the following corollary
\begin{cor}\label{Cor:Cof}
Suppose that \BS{} implies having uncountable cofinality.
Let $G$ and $(G_x)_{x\in X}$ be non-trivial groups and let $H$ be a group acting on $X$. Then
\begin{enumerate}
\item $\bigoplus_{x\in X}G_x$ has property~\BS{} if and only if all the $G_x$ have property~\BS{} and $X$ is finite,
\item If $G\wr_XH$ has property~\BS, then $H$ acts on $X$ with finitely many orbits.
\end{enumerate}
\end{cor}

While the statement (and the proof) of Corollary~\ref{Cor:Cof}.1 is expressed in terms of uncountable cofinality, it is also possible to state it and prove it for a subcategory \CatS{} of \PMet{} without a priori knowing if \BS{} is stronger than having uncountable cofinality.
The main idea is to find a ``natural'' \CatS-space on which $G=\bigoplus_{i\geq 1}G_i$ acts. For example, for (reflexive) Banach, Hilbert and $L^p$ spaces, one can take $\bigoplus_{i\geq 1}\ell^p(G_i)$. For connected median graphs, one takes the connected component of $\{1_{G_1},1_{G_2},\dots\}$ in $\powerset{\bigsqcup_{i\geq 1} G_i}$.
For (real) trees, it is possible to put a forest structure on $\powerset{\bigsqcup_{i\geq 1} G_i}$ in the following way.
For $E\in\powerset{\bigsqcup_{i\geq 1} G_i}$, and for each $i$ such that $E\cap G_j$ is empty for all $j\leq i$, add an edge from $E$ to $E\cup\{g\}$ for each $g\in G_i$. The graph obtained this way is a $G$-invariant subforest of the median graph on $\powerset{\bigsqcup_{i\geq 1} G_i}$.
For Corollary~\ref{Cor:Cof}.2 we also need that the corresponding structure is invariant by the action of $H$, which is the case of the above examples, except for the tree structure.

It is also possible to give a proof of Corollary~\ref{Cor:Cof}.1 using axioms similar to (A1)-(A3).
More precisely, we need a variation of (A3) for countable Cartesian products (for morphisms we only ask that $\bigoplus_{i\in\N} \Aut_\CatS(X_i)\subseteq\Aut_\CatS(\bigoplus_{i\in\N} X_i)$) and a strong version of (A1) saying that there exists an universal bound $M$ such that any non-trivial group $G$ acts on some \CatS-space moving a point at distance at least $M$.
The axiomatization of Corollary~\ref{Cor:Cof}.2 is a little more complex.
However, since in the following we will use Corollary~\ref{Cor:Cof} only for (real) trees, which do not have Cartesian powers, we will not elaborate on the details and let the proof to the interested reader.
Instead, we will give an axiomatic proof of the following variation of Corollary~\ref{Cor:Cof}.2.
%
%
%
%
%
\begin{lem}\label{Lemma:XFinite}
Suppose that \CatS{} has non-trivial group actions and satisfies axiom~(A3).
Let $G$ and $H$ be two groups with $G$ non-trivial and let $X$ be a set on which $H$ acts.
If $G\wr_XH$ has \BS, then $X$ is finite.
\end{lem}
\begin{proof}
We will prove that if $X$ is infinite, then $G\wr_XH$ does not have property~\BS. Suppose that $X$ is infinite.
By non-trivial groups actions, there exists an \CatS-space $Y$ on which $G$ acts non-trivially by moving some element $y_0$ to another element $z_0\neq y_0$.
Let $\bigoplus_XY$ be the corresponding Cartesian power and $f_0$ be the constant function $f_0(x)=y_0$.
By assumption, the natural action of $G\wr_XH$ on $\bigoplus_XY$ is by \CatS-automorphisms.
Since $X$ is infinite, it contains a countable subset $I=\{i_1,i_2,\dots\}$.
For every integer $n$, the function
\[
	f_n(x)\coloneqq
	\begin{cases}
		z_0 & \textnormal{if } x=i_m, m\leq n\\
		y_0 & \textnormal{otherwise}
	\end{cases}
\]
is in the $G\wr_XH$-orbit of $f_0$. By axiome (A3) this orbit is unbounded and $G\wr_XH$ does not have property~\BS.
%
%
\end{proof}
%
%

As a direct corollary, we obtain that if property \BS{} implies property \BS' for some \CatS' with non-trivial group actions and (A3) (example: \BS'=\FW), then the conclusion of Lemma~\ref{Lemma:XFinite} holds even if \CatS{} might not satisfy its premises.
Conversely, it follows from Theorem~\ref{Thm:FACK} and Proposition~\ref{Prop:UncCoun} that Lemma~\ref{Lemma:XFinite} do not holds for property~\FR, property \FA{} or having uncountable cofinality.
%
%
%
%
%

We now turn our attention to properties that behave well under finite Cartesian products in the sense of axiom (A2).
We first describe the comportment of property~\BS{} under finite index subgroups.
\begin{lem}\label{Lemma:Subgroup}
Let $G$ be a group and let $H$ be a finite index subgroup. Then
\begin{enumerate}
\item
If $H$ has property~\BS, then so has~$G$,
\item
If \CatS{} satisfies (A2) and $G$ has property~\BS, then $H$ has property~\BS.
\end{enumerate}
\end{lem}
\begin{proof}
Suppose that $G$ does not have \BS{} and let $X$ be an \CatS-space on which $G$ acts with an unbounded orbit $\orbite$.
Then $H$ acts on $X$ and $\orbite$ is a union of at most $[G:H]$ orbits. This directly implies that $H$ has an unbounded orbit and therefore does not have \BS.

On the other hand, suppose that $H\leq G$ is a finite index subgroup of G without property~\BS.
Let $\alpha\colon H\curvearrowright X$ be an action of $H$ on an \CatS-space $(X,d_X)$ such that there is an unbounded orbit $\orbite$.
Similarly to the classical theory of representations of finite groups, we have the induced  action $\Ind_H^G(\alpha)\colon G \curvearrowright X^{G/H}$ on the set $X^{G/H}$. Since $H$ has finite index, $X^{G/H}$ is an \CatS-space and the action is by \CatS-automorphisms. On the other hand, the subgroup $H\leq G$ acts diagonally on $X^{G/H}$, which implies that $\diag(\orbite)$ is contained in a $G$-orbit.
Since $\diag(\orbite)$ is unbounded, $G$ does not have property~\BS.

For readers that are not familiar with representations of finite groups, here is the above argument in more details.
Let $(f_i)_{i=1}^n$ be a transversal for $G/H$.
The natural action of $G$ on $G/H$ gives rise to an action of $G$ on $\{1,\dots,n\}$.
Hence, for any $g$ in $G$ and $i$ in $\{1,\dots,n\}$ there exists a unique $h_{g,i}$ in $H$ such that $gf_i=f_{g.i}h_{g,i}$. That is, $h_{g,i}=f_{g.i}^{-1}gf_i$.
We then define $g.(x_1,\dots,x_n)\coloneqq(h_{g,g^{-1}.1}.x_{g^{-1}.1},\dots,h_{g,g^{-1}.n}.x_{g^{-1}.n})$. This is indeed an action by \CatS-auto\-mor\-phisms on $X^{G/H}$ by Condition~\ref{Item:Product} of Definition~\ref{Def:Cartesian}.
Moreover, every element $h\in H$ acts diagonally by $h.(x_1,\dots,x_n)=(h.x_1,\dots,h.x_n)$.
In particular, this $G$-action has an unbounded orbit.
\end{proof}
We now prove one last lemma that will be necessary fo the proof of Theorem~\ref{Thm:Main}.
\begin{lem}\label{Lemma:Unboundedness}
Suppose that \CatS{} satisfies (A2). If $X$ is finite and $G\wr_XH$ has property~\BS, then $G$ has property~\BS.
\end{lem}
\begin{proof}
Suppose that $G$ does not have \BS{} and let $(Y,d_Y)$ be an \CatS-space on which $G$ acts with an unbounded orbit $G.y$.
Then $(Y^X,d)$ is an \CatS-space and we have the \emph{primitive action} of the wreath product $G\wr_XH$ on $Y^X$:
\[
	\bigl((\varphi,h).\psi\bigr)(x)=\varphi(h^{-1}.x).\psi(h^{-1}.x).
\]
By Condition~\ref{Item:Product} of Definition~\ref{Def:Cartesian}, this action is by \CatS-automorphisms.
The orbit $G.y$ embeds diagonally and hence $\diag(G.y)$ is an unbounded subset of some $G\wr_XH$-orbit, which implies that $G\wr_XH$ does not have property~\BS.
\end{proof}
It is also possible to derive Lemma~\ref{Lemma:Unboundedness} directly from Lemma~\ref{Lemma:Subgroup}.2, with a more algebraic proof.
Indeed, using the notation and hypothesis of Lemma~\ref{Lemma:Unboundedness}, let $H'$ be the kernel of the action of $H$ on $X$ and $\pi\colon G\wr_XH\to H$ be the canonical projection. Then $\pi^{-1}(H')\cong \bigoplus_XG\oplus H'$ is a finite index subgroup of $G\wr_XH$ and hence has property~\BS.
Since $G$ is a quotient of $\bigoplus_XG\oplus H'$ we conclude that it also has property~\BS.


We now proceed to prove Proposition~\ref{Prop:UncCoun}.
As for Lemma~\ref{Lemma:Cof}, it is also possible to prove it using the characterization of uncountable cofinality in terms of subgroups, in which case it is an easy exercise, but we will only give a proof using the characterization in terms of actions on ultrametric spaces.

%
%
\begin{proof}[Proof of Proposition~\ref{Prop:UncCoun}]
By Corollary~\ref{Cor:Wreath} and Lemma~\ref{Lemma:Cof} we already know that if $G \wr_X H$ has uncountable cofinality, then $H$ has uncountable cofinality and it acts on $X$ with finitely many orbits.
We will now prove that if $G \wr_X H$ has uncountable cofinality so does $G$.
Let us suppose that $G$ has countable cofinality. By Lemma~\ref{Lemma:CofSub}, there exists an ultrametric $d$ on $G$ such that the action of $G$ on itself by left multiplication has an unbounded orbit.
But then we have the primitive action of the wreath product $G\wr_XH$ on $G^X\cong\prod_XG$, which preserves $\bigoplus_XG$.
It is easy to check that the map $d_\infty\colon\bigoplus_XG\times\bigoplus_XG\to\R$ defined by $d_\infty(\psi_1,\psi_2)\coloneqq\max\setst{d\bigl(\psi_1(x),\psi_2(x)\bigr)}{x\in X}$ is a $G\wr_XH$-invariant ultrametric.
Finally, let $g_0\in G$ be an element of unbounded $G$-orbit for $d$ and let $x_0$ be any element of $X$.
For $g$ in $G$ and $x$ in $X$, we define the following analog of Kronecker's delta functions
\begin{equation*}
\delta_x^g (y) \coloneqq
\begin{cases}
g & \textnormal{if }y = x \\
1 & \textnormal{if } y \neq x.
\end{cases}
\end{equation*}
Then we have $(\delta_{x_0}^g,1).\delta_{x_0}^{g_0}=\delta_{x_0}^{gg_0}$ and hence $d_\infty(\delta_{x_0}^{g_0},\delta_{x_0}^{gg_0})=d(g_0,gg_0)$ is unbounded.

Suppose now that both $G$ and $H$ have uncountable cofinality and that $H$ acts on $X$ with finitely many orbits. We want to prove that $G\wr_XH$ has uncountable cofinality.

Let $(Y,d)$ be an ultrametric space on which $G\wr_XH$ acts.
Then $H$ and all the $G_x$ act on $Y$ with bounded orbits.
Let $H.x_1,\dots, H.x_n$ be the $H$-orbits on~$X$ and let $y$ be any element of $Y$.
Then $H.y$ has finite diameter $D_0$ while $G_{x_i}.y$ has finite diameter $D_i$.
For any $x\in X$, there exists $1\leq i\leq n$ and $h\in H$ such that $x=h.x_i$.
We have
\begin{align*}
	d\bigl((\delta_{x_i}^g,h^{-1}).y,y\bigr)&\leq\max\{d\bigl((\delta_{x_i}^g,h^{-1}).y,(\delta_{x_i}^g,1).y\bigr),d\bigl((\delta_{x_i}^g,1).y,y\bigr)\}\\
	&=\max\{d\bigl((1,h^{-1}).y,y\bigr),d\bigl((\delta_{x_i}^g,1).y,y\bigr)\}\\
	&\leq \max\{D_0,D_i\},
\end{align*}
which implies that the diameter of $G_{x_i}h^{-1}.y$ is bounded by $\max\{D_0,D_i\}$.
But $G_{x_i}h^{-1}.y$ has the same diameter as $hG_{x_i}h^{-1}.y=G_{h.x_i}.y=G_x.y$.

On the other hand, the diameter of $\bigoplus_XG.y$ is bounded by the supremum of the diameters of the $G_{x}.y$, and hence bounded by $\max\{D_0,D_1,\dots,D_n\}$.
Finally, for $(\varphi,h)$ in $G\wr_X H$ we have
\begin{align*}
	d\bigl(y,(\varphi,h).y\bigr)&\leq\max\{d\bigl(y,(\varphi,1).y\bigr),d\bigl((\varphi,1).y,(\varphi,h).y\bigr)\}\\
	&=\max\{d\bigl(y,(\varphi,1).y\bigr),d\bigl(y,(1,h).y\bigr)\}\\
	&\leq\max\{\max\{D_0,D_1,\dots,D_n\},D_0\}.
\end{align*}
That is, the diameter of $G\wr_XH.y$ is itself bounded by $\max\{D_0,D_1,\dots,D_n\}$, which finishes the proof.
\end{proof}
While the fact that trees do not have Cartesian powers is an obstacle to our methods, we still have a weak version of Theorem~\ref{Thm:Technic} for properties~\FA{} and~\FR.
Before stating it, remind that we already know, by Proposition~\ref{Prop:UncCoun}, the behavior of uncountable cofinality under wreath products.
On the other hand, we have the following result:
\begin{lem}\label{Lemma:Z}
The group $G\wr_XH$ has no quotient isomorphic to $\Z$ if and only if both $G$ and $H$ have no quotient isomorphic to $\Z$.
\end{lem}
\begin{proof}
The desired result follows from $(G\wr_XH)^{\ab}\cong \bigoplus_{X/H}(G^{\ab})\times H^{\ab}$ and the claim that a direct sum $\bigoplus_{y\in Y}K_y$ has a quotient isomorphic to $\Z$ if and only if at least one of the factor has a quotient isomorphic to $\Z$.
Indeed, one direction of the claim is trivial.
For the other direction, remind that $K$ does not project onto $\Z$ if and only if any action of $K$ by orientation preserving isomorphisms on $Z$, the $2$-regular tree, has bounded orbits. But the only possibility for such an action to have bounded orbits is to be trivial.
If none of the $K_y$ projects onto $\Z$, all their actions on $Z$ are trivial and so is any action of $\bigoplus_{y\in Y}K_y$, which can therefore not project onto $\Z$.
\end{proof}
By Corollary~\ref{Cor:Wreath}, Proposition~\ref{Prop:UncCoun} and Lemma~\ref{Lemma:Z}, we directly obtain the following partial version of Theorem~\ref{Thm:FACK}.
\begin{prop}\label{Prop:WRFA}
Let $G$ and $H$ be two groups with $G$ non-trivial and $X$ a set on which $H$ acts.
Then
\begin{enumerate}
\item If $G\wr_XH$ has property~\FA{} (respectively property~\FR), then $H$ has property~\FA{} (respectively property~\FR), $H$ acts on $X$ with finitely many orbits, $G$ has no quotient isomorphic to $\Z$ and $G$ has uncountable cofinality,
\item If both $G$ and $H$ have no quotient isomorphic to $\Z$, have uncountable cofinality and $H$ acts on $X$ with finitely many orbits, then $G\wr_XH$ has no quotient isomorphic to $\Z$ and has uncountable cofinality,
\item If both $G$ and $H$ have property~\FA{} (respectively property~\FR) and $X$ is finite, then $G\wr_XH$ has property~FA (respectively property~\FR).
\end{enumerate}
\end{prop}
Moreover, by using Corollary~\ref{Cor:Cof} we can get ride of the \emph{finitely many orbits} hypothesis in Theorem~\ref{Thm:FACK} in order to obtain Theorem~\ref{Thm:FAFiniteOrbits}.

%
%
%
%
%
%
%
%
%
%

%
%
%
%
%
%
%
%
%
%
\end{document}